\documentclass[a4paper, 11pt]{article}
\usepackage{bbm}
\usepackage{amsfonts}
\usepackage{amsmath}
\usepackage{amsthm,mathrsfs}
\usepackage{amsfonts,amssymb}
\usepackage{float}
\usepackage{pdfsync}
\usepackage{graphicx}
\usepackage{caption}
\usepackage{hyperref}
\allowdisplaybreaks[4]

\newtheorem{theorem}{Theorem}[section]

\newtheorem{lemma}{Lemma}[section]
\newtheorem{definition}{Definition}[section]

\newtheorem{remark}{Remark}[section]
\theoremstyle{remark}
\theoremstyle{corollary}
\newtheorem{corollary}{Corollary}[section]
\theoremstyle{remark}
\theoremstyle{step}

\numberwithin{equation}{section}

\def\dfrac{\displaystyle\frac}

\def\supp{{\mathop\mathrm{\,supp\,}}}

\def\p{\partial}

\def\v{\varphi}

\def\supp{{\mathop\mathrm{\,supp\,}}}
\def\rr{{\mathbb R}}

\begin{document}
\title{{Critical Fujita exponent for a semilinear heat equation with degenerate coefficients}
\footnote
{Supported by the National Natural Science Foundation of China (No.11771023).}}

\date{}\maketitle
\vspace{-2cm}
\begin{center}
Xi Hu, Lin Tang

\end{center}
{\bf Abstract:} We prove the existence of a critical Fujita exponent for a non-homogeneous semilinear heat equation which involves degenerate coefficients. More precisely, in order to give a rather complete theory, we focus on two types of weights $w(x)=|x_1|^a$ or $w(x)=|x|^b$ where $a, b>0$ in a suitable range. The coefficients under consideration admit either a singularity at the origin or a line of singularities. In the latter case, the problem is related to the fractional Laplacian.
\vspace{0.2cm}
\vspace{0.2cm}

\section{Introduction}
Beginning with the classical paper by Fujita \cite{FH}, critical exponents for the global existence of solutions (not only positive ones but also sign-changing ones) were established for many classes of evolution problems. It seems almost impossible to make complete list of this topics. So we only refer a part of them for instance \cite{FU, GL, IK0, IK1, IK2, LH, MI, PR, WF1, WF2} and the references therein. Among others, in particular, Fujishima, Kawakami and Sire \cite{FU} showed the Fujita exponent for the following problem
\begin{equation*}\label{1}
\left \{
   \begin{array}{ll}
\p_t u-\mathrm{div}(w(x)\nabla u)=u^p,\quad ~&x\in\rr^n,\quad t>0,\\\tag{1.1}
u(x,0)=\v(x)\geq 0, & x\in \rr^n,
 \end{array}
 \right.
\end{equation*}
where the coefficient $w$ is either $w(x)=|x_1|^a$ with $a\in [0,1) $ if $n=1,2$ and $a\in [0, 2/n)$ if $n\geq 3$, or $w(x)=|x|^b$ with $b\in [0,1)$. They proved that, if $1<p\leq 1+\frac{2-\alpha}{n}$ where $
\alpha\in \{a,b\}$,  then global-in-time positive solutions of (1.1) do not exist. Furthermore they showed that, if $p> 1+\frac{2-\alpha}{n}$ where $\alpha\in \{a,b\}$, then, for some ``small'' initial function $\v$, there exists a global-in-time solution of (1.1).  The problem (1.1) is related to the fractional Laplacian $\p_tu+(-\Delta)^su=u^p, x\in \rr^n , t>0$, through the Caffarelli-Silvestre extension \cite{CS} and is a first attempt to develop a parabolic theory in this setting.

On the other hand, for the modified nonlocal equation $$(\p_t -\Delta)^su=u^p, \quad x\in \rr^n,\quad t>0,\eqno(1.2)$$
we consider the problem (1.2) from another point of view, invoking Caffarelli-Silvestre extension, one  gets a degenerate elliptic equation in the half space with a nonlinear dynamical boundary condition. In order to analyze this problem, one needs to compare this problem with the degenerate parabolic equation with nonlinear boundary condition in the half space. From this point of view, before treating this half space problem, we consider the similar problem in the whole space. More precisely,
we consider the following problem
\begin{equation*}\label{1}
\left \{
   \begin{array}{ll}
\p_t u-w^{-1}\mathrm{div}(w(x)\nabla u)=u^p,\quad ~&x\in\rr^n,\quad t>0\\\tag{1.3}
u(x,0)=u_0(x)\geq 0, & x\in \rr^n,
 \end{array}
 \right.
\end{equation*}
where the coefficient $w$ is either $w(x)=|x_1|^a$ with $a\in [0,1)$, or $w(x)=|x|^b$ with $b\in [0,n)$, $n\geq 1$, and $p>1.$

The purpose of this paper is to develop a global-in-time existence theory of mild solutions for the problem (1.3). We prove that there exists a critical exponent for the global existence of positive solutions of problem (1.3), the so-called Fujita exponent.

We first give the definition of a solution to (1.3).

Let $\Gamma=\Gamma (x,y,t)$ be the fundamental solution of $$\p_t v-w^{-1}\mathrm{div}(w(x)\nabla v)=0,\quad x\in \rr^n,\quad t>0,$$ with a pole at $(y,0).$ Under the condition either $w(x)=|x_1|^a$ with $a\in [0,1)$, or $w(x)=|x|^b$ with $b\in [0,n)$, this fundamental solution $\Gamma$ satisfies the mass conservation property, the semigroup one and suitable Gaussian estimates [see $\mathrm{(K1)}$--$\mathrm{(K3)}$ in Section 2]. By $\Gamma $, we define the solution of (1.3) as follows.
\begin{definition}
Let $u_0$ be a nonnegative measurable function in $\rr^n$. Let $T\in (0,\infty]$ and $u$ be a nonnegative measurable function in $\rr^n\times(0,T)$ such that $u\in L^\infty(0,T:L^\infty(\rr^n))$. Then we call $u$ a solution of  $(1.3)$ in $\rr^n\times (0,T)$ if $u$ satisfies
$$u(x,t)=\int_{\rr^n}\Gamma(x,y,t)u_0(y)w(y)dy+\int^t_0\int_{\rr^n}\Gamma(x,y,t-s)u(y,s)^pw(y)dyds<\infty \eqno(1.4)$$
for almost all $x\in \rr^n$ and $t\in (0,T)$. In particular, we call $u$ a global-in-time solution of $(1.3)$ if $u$ is a solution of $(1.3)$ in $\rr^n\times(0,\infty).$
\end{definition}
\begin{remark}
For the case of the semilinear heat equation, that is, $(1.3)$ with either $a=0$ or $b=0$, if $u_0\in L^\infty(\rr^n)$, then there exists a local-in-time solution $u$ of $(1.3)$ in $\rr^n\times(0,T)$ for some $T>0$ satisfying $(1.4)$; see $\cite{HW}$ and $\cite{WF}$.
\end{remark}

The previous definition is the well-known class of mild solutions and is natural to deal with parabolic problems. In our context, it is vital to notice that due to the non-homogeneity of the operator, the fundamental solution is not translation-invariant. Furthermore, there is no explicit expression of it, though bounds are known. This makes the theory harder.

On the other hand, we discuss the features of the weight $w(x)$. In this paper, we do not consider general weights since it is very complicated in this case to give precise results as our purpose is. We will consider two types of weights. The first one is $w(x)=|x_1|^a$ which belongs to the class $A_2$ of Muckenhoupt functions \cite{MB} if $a\in [0,1)$ and exhibits singularities along the line $x_1=0$. The other weight under consideration is $w(x)=|x|^b$ which is $A_2$ for $b\in [0,n)$ and exhibits a singularity at the origin $x=0$.

Before stating our main results, we first introduce some notations. For any $x\in \rr^n$ and $R>0$, we set $B(x,R):=\{y\in \rr^n:|x-y|<R\}$. Let $(\rr^n, w)$ be a measure space,  where $w$ belongs to the class $A_2$ of Muckenhoupt functions. For  any $1\leq r\leq \infty$, given $f$ a measurable function on a measure space $(\rr^n, w)$, we define the $L^r(w)$ norm of a function $f$ by $$\|f\|_{L^r(w)}:=\left(\int_{\rr^n}|f(x)|^rw(x)dx\right)^{\frac{1}{r}}.$$ For any measurable function $f$ on a measure space $(\rr^n, w)$,
$$\mu_f(\lambda):=w\left(\{x: |f(x)|>\lambda\}\right),\quad \lambda\geq 0,\eqno(1.5)$$
is the distribution function of $f$. We denote the non-increasing rearrangement of $f$ by $$f^*(s):=\inf \{\lambda>0: \mu_f(\lambda)\leq s\}.\eqno(1.6)$$
The spherical rearrangement of $f$ is denoted by $$f^{\sharp}(x):=f^*(c_n |x|^n),$$
where $c_n$ is the volume of the unit ball in $\rr^n$.

Then, for any $1\leq r\leq \infty$ and $1\leq \sigma\leq \infty$, we define the Lorentz space $L^{r,\sigma}(w)$ by
$$L^{r,\sigma}(w):=\{f: f~\text{is measurable function on a measure space $(\rr^n, w)$},~\|f\|_{L^{r,\sigma}(w)}<\infty\},$$
where
\begin{equation*}\label{1}
\|f\|_{L^{r,\sigma}(w)}:=\left \{
   \begin{array}{ll}
\displaystyle\left(\int^\infty_0\left[s^{\frac{1}{r}}f^*(s)\right]^\sigma\dfrac{ds}{s}\right)^{\frac{1}{\sigma}}\quad & \text{if}~1\leq \sigma<\infty,\\\tag{1.7}
\displaystyle{\sup_{s>0}}~s^{\frac{1}{r}}f^*(s) & \text{if}~\sigma=\infty.
\end{array}
\right.
\end{equation*}
The Lorentz $L^{r,\sigma}(w)$ is a Banach space and the following holds (see \cite{GR}):
\begin{itemize}
\item Let $1<r<\infty$. Then $f\in L^{r,\infty}(w)$ if and only if $$0\leq f^\sharp(x)\leq C_1|x|^{-n/r},\quad x\in \rr^n,\eqno(1.8)$$
for some constant $C_1$;
\item For $1<r<\infty$, it follows
\begin{equation*}\label{1}
\|f\|_{L^{r,\sigma}(w)}:=\left \{
   \begin{array}{ll}
\displaystyle r^{\frac{1}{\sigma}}\left(\int^\infty_0\left[s \mu_f(s)^{\frac{1}{r}}\right]^\sigma\dfrac{ds}{s}\right)^{\frac{1}{\sigma}}\quad & \text{if}~1\leq \sigma<\infty,\\\tag{1.9}
\displaystyle{\sup_{s>0}}~s \mu_f(s)^{\frac{1}{r}} & \text{if}~\sigma=\infty.
\end{array}
\right.
\end{equation*}
\item $L^{r,r}(w)=L^r(w)$ if $1<r< \infty$, $L^{\infty,\infty}(w)=L^\infty(w)=L^\infty$ and $L^{r,\sigma_1}(w)\subset L^{r,\sigma_2}(w)$ if $1\leq r\leq \infty$ and $1\leq \sigma_1\leq \sigma_2\leq \infty;$
\item Let $1\leq r_0\leq r\leq r_1\leq \infty$ be such that $$\frac{1}{r}=\frac{1-\theta}{r_0}+\frac{\theta}{r_1}\quad \text{for}~\theta\in [0,1].$$
Then it holds that $$\|f\|_{L^{r,\infty}(w)}\leq \|f\|^{1-\theta}_{L^{r_0,\infty}(w)}\|f\|^\theta_{L^{r_1,\infty}(w)},\quad f\in L^{r_0,\infty}(w)\cap L^{r_1,\infty}(w);\eqno(1.10)$$
\item Let $1\leq r_1\leq \infty$ and $r_2$ be the H\"{o}lder conjugate number of $r_1$, that is, $1/r_1+1/r_2=1.$
Then it holds that $$\|fg\|_{L^1(w)}\leq \|f\|_{L^{r_1,1}(w)}\|g\|_{L^{r_2,\infty}(w)},\quad f\in L^{r_1,1}(w),\quad g\in L^{r_2,\infty}(w).\eqno(1.11)$$
\end{itemize}

We now state the main results of this paper but several explanations are in order. In most of the parabolic problem dealing with homogeneous equations, an important role is played by the fundamental solution. It happens that one can deduce several strong results as soon as one has an explicit form of the fundamental solution, allowing to get estimates for the function and its derivatives; see \cite{IK1, IK3, IK4}. In problem (1.3), even if the coefficients are rather simple, such an explicit form is unavailable.

In what follows, we set $$p_*(\alpha):=1+\frac{2}{n+\alpha}\quad\text{for}~\alpha\in\{a,b\}.$$
Moreover, we assume either
\begin{itemize}
\item[$\mathrm{(A)}$] $w(x)=|x_1|^a$ with $a\in [0,1)$, or
\item[$\mathrm{(B)}$] $w(x)=|x|^b$ with $b\in [0, n)$.
\end{itemize}

The first theorem is concerned with the nonexistence of positive global-in-time solutions of (1.3).

\begin{theorem}
Assume either $\mathrm{(A)}$ or $\mathrm{(B)}$. Let $\alpha$ be such that $\alpha=a$ for the case $\mathrm{(A)}$ and $\alpha=b$ for the case $\mathrm{(B)}$.
Assume $1<p\leq p_*(\alpha).$ Then problem $(1.3)$ has no positive global-in-time solutions.
\end{theorem}

In second theorem we give a sufficient condition for the existence of  global-in-time solutions of (1.3).

\begin{theorem}
Assume either $\mathrm{(A)}$ or $\mathrm{(B)}$. Let $\alpha$ be such that $\alpha=a$ for the case $\mathrm{(A)}$ and $\alpha=b$ for the case $\mathrm{(B)}$.
Assume $p>p_*(\alpha)$. Set $$r_*:=\dfrac{n+\alpha}{2}(p-1)>1.\eqno(1.12)$$
Then the following holds: \begin{itemize} \item[$\mathrm{(i)}$] There exists a positive constant $\delta$ such that, for any $u_0\in L^\infty\cap L^{r_*,\infty}(w)$ with $$\|u_0\|_{L^{r_*,\infty}(w)}<\delta,\eqno(1.13)$$
a unique global-in-time solution $u$ of $(1.3)$ exists and it satisfies $$\sup_{t>0}(1+t)^{\frac{n+\alpha}{2}\left(\frac{1}{r_*}-\frac{1}{q}\right)}\|u(t)\|_{L^{q,\infty}(w)}<\infty, \quad r_*\leq q\leq\infty.\eqno(1.14)$$
\item[$\mathrm{(ii)}$] Let $1\leq r\leq r_*$. Then there exists a positive constant $\delta$ such that, for any $u_0\in L^\infty\cap L^r(w)$ with
$$\|u_0\|^{\frac{r}{r_*}}_{L^r(w)}\|u_0\|^{1-\frac{r}{r_*}}_\infty<\delta,\eqno(1.15)$$
a unique  global-in-time solution $u$ of $(1.3)$ exists and it satisfies
$$\sup_{t>0}(1+t)^{\frac{n+\alpha}{2}\left(\frac{1}{r}-\frac{1}{q}\right)}\|u(t)\|_{L^{q}(w)}<\infty,\quad r\leq q\leq\infty.\eqno(1.16)$$
\end{itemize}
\end{theorem}
\begin{remark}
As far as the regularity of the mild solutions constructed in this paper is concerned, Chiarenza and Serapioni $\cite{CH1, CH}$ considered degenerate parabolic equations with $A_2$ weights. However, their starting point are weak solutions. In order to upgrade our mild solutions to weak solutions one needs gradient bounds on the fundamental solution $\Gamma$, which are unavailable.
\end{remark}

As a direct consequence of Theorem 1.2, we get:
\begin{corollary}
Let $\alpha\in\{a,b\}$. Assume $p>p_*(\alpha)$. Then there exists a positive constant $\delta$ such that, if $$|u_0(x)|\leq \dfrac{\delta}{1+|x|^{2/(p-1)}},\quad x\in\rr^n,\eqno(1.17)$$
then a unique global-in-time solution $u$ of $(1.3)$ exists and it satisfies $(1.14)$.
\end{corollary}
\begin{remark}
If $u_0(x)$ satisfies $(1.17)$, then it follows from $(1.8)$ that $u_0\in L^{r_*,\infty}(w)$. On the other hand, if $u_0(x)=O(|x|^{-2/(p-1)})$ as $|x|\to \infty$, then $u_0\notin L^{r_*}(w)$. This is a clear advantage in using $L^{r,\infty}(w)$ spaces instead of the classical $L^r(w)$ spaces.
\end{remark}
\begin{remark}
By Theorems $1.1$ and $1.2$ we see that $p_*(\alpha)$ is the Fujita exponent for problem $(1.3)$. In fact, if $\alpha=0$, then $p_*(0)=1+2/n$, which is the Fujita exponent for $(1.3)$ with $w(x)\equiv1.$
\end{remark}
\begin{remark}
In order to prove the regularity of the solution satisfying $(1.4)$, in spite of the case of the semilinear heat equation, we need suitable bounds for the derivatives and the translation-invariant property of the fundamental solution; see $\cite{FR}$ and $\cite{IK1}$.
However, unfortunately, it seems that they have been still left open. On the other hand, under our definition, for the sake of proving the existence/nonexistence of global-in-time solutions of $(1.3)$, we only need properties $\mathrm{(K1)}$--$\mathrm{(K3)}$ and decay estimates, which are given in Lemma $2.2.$
\end{remark}

The rest of this paper is organized as follows: In Section 2 we give some preliminary results on the fundamental solution and $w(B(x,r))$. In particular, in Lemma 2.1, we show the lower and upper estimates of $w(B(x,r))$. Furthermore, the decay estimates of the fundamental solution are established in Lemma 2.2. In Section 3 we prove the Theorem 1.1, which means that problem (1.3) has no positive global-in-time solutions. Finally, we prove the uniqueness and local existence of solutions of (1.3) and then obtain Theorem 1.2 in Section 4.

\section{Preliminaries}
A vital tool in our arguments is based on the use of the fundamental solution of the operator $\p_t-w^{-1}\mathrm{div}(w(x)\nabla\cdot)$. As already mentioned due to the non-homogeneity of the operator, an explicit formula is not known but bounds are available as follows.

Under condition either $(\mathrm{A})$ or $(\mathrm{B})$, the weights $w$ belonging to the class  $A_2$ of Muckenhoupt functions, the fundamental solution $\Gamma=\Gamma(x,y,t)$ has the following properties (see \cite{CU} and \cite[Section 4]{IK}):
\begin{itemize}
\item[$(\mathrm{K1})$] $\displaystyle\int_{\rr^n} \Gamma (x,y,t)w(x)dx=\int_{\rr^n}\Gamma(x,y,t)w(y)dy=1$  ~for $x,y\in \rr^n$ and $t>0;$
\item[$(\mathrm{K2})$] $\displaystyle\Gamma (x,y,t)=\int_{\rr^n}\Gamma (x,\xi,t-s)\Gamma(\xi,y,s)w(\xi)d\xi$  ~for $x,y\in \rr^n$ and $t>s>0;$
\item[$(\mathrm{K3})$] There exist positive constants $c_*$ and $C_*$ depending only on $n$ and $\alpha\in \{a,b\}$ such that
\begin{align*}
   &\dfrac{c_*}{\sqrt{w(B(x,\sqrt{t}))}\sqrt{w(B(y,\sqrt{t}))}}\mathrm{exp}\left(-\dfrac{|x-y|^2}{c_*t}\right)\leq\Gamma(x,y,t)\\  &\leq\dfrac{C_*}{\sqrt{w(B(x,\sqrt{t}))}\sqrt{w(B(y,\sqrt{t}))}}\mathrm{exp}\left(-\dfrac{|x-y|^2}{C_*t}\right)
\end{align*}
for $x,y\in \rr^n$ and $t>0$. Here $w(B(x,\sqrt{t})):=\int_{B(x,\sqrt{t})}w(z)dz.$
\end{itemize}

By $\mathrm{(A)}$ and $(\mathrm{B})$, we state a lemma on lower and upper estimates of $w(B(x,r))$. In what follows, we denote  positive constants $C$ and $C'$ depending only on  $n$ and $\alpha\in \{a,b\}$ and they may have different values also within the same  line.
\begin{lemma}
The following holds:
\begin{itemize}
\item[$\mathrm{(i)}$] Let $w(x)=|x_1|^a$ and assume condition $\mathrm{(A)}$. Then there exist positive constants $C$ and $C'$ depending only on $n$ and $a$ such that
  \begin{equation}
  w(B(x,r))\geq Cr^{n+a}\end{equation}
and
  \begin{equation}
  w(B(x,r))\leq C'\left\{
  \begin{aligned}
  &r^n|x_1|^a\qquad\text{if}~0<r\leq |x_1|,\\
  &r^{n+a}\quad\qquad \text{if}~r\geq |x_1|,
  \end{aligned}
  \right.
  \end{equation}
  for all $x\in \rr^n$ and $r>0.$
\item[$\mathrm{(ii)}$] Let $w(x)=|x|^b$ and assume $\mathrm{(B)}$. Then there exist positive constants $C$ and $C'$ depending only on $n$ and $b$ such that
\begin{equation}w(B(x,r))\geq Cr^{n+b}\end{equation}
and \begin{equation}
w(B(x,r))\leq C'\left\{
\begin{aligned}
&r^n|x|^b\qquad ~\text{if}~0<r\leq |x|,\\
&r^{n+b}\qquad \quad\text{if}~r\geq |x|,
 \end{aligned}
 \right.
\end{equation}
for all $x\in \rr^n$ and $r>0.$
\end{itemize}
\end{lemma}
\begin{proof}
We first prove assertion $\mathrm{(i)}$. Since $w(y)=|y_1|^a$ is monotonically increasing function with respect to the distance from the origin, we have
\begin{align*}
w(B(x,r))&=\int_{B(x,r)}|y_1|^ady\geq \int_{B(0,r)}|y_1|^ady\\
&=\int_{-r}^r|y_1|^a\left(\int_{|y'|_{n-1}<\sqrt{r^2-y^2_1}}dy'\right)dy_1\\
&=2\omega_{n-1}\int_{0}^ry_1^a(r^2-y^2_1)^{\frac{n-1}{2}}dy_1\\
&=\omega_{n-1}r^{n+a}\int_{0}^1\xi^{\frac{a-1}{2}}(1-\xi)^{\frac{n-1}{2}}d\xi\\
&=\omega_{n-1}B\left(\frac{a+1}{2},\frac{n+1}{2}\right)~r^{n+a}
\end{align*}
for all $x\in \rr^n$ and $r>0$, where $y=(y_1,y')\in \rr^n$, $|\cdot|_{n-1}$ denotes the usual Euclidean norm in $\rr^{n-1}$, $\omega_{n-1}$ denotes the volume of the unit ball in $\rr^{n-1}$ and $B(\cdot,\cdot)$ denotes the beta function. This implies (2.1).

On the other hand, since $w(y)=|y_1|^a$ depends only on $y_1$ variable, for any $x=(x_1,x')\in \rr\times \rr^{n-1}$, we can choose a point $x_*=(x_1,0)$ such that
$$\int_{B(x,r)}w(y)dy=\int_{B(x_*,r)}w(y)dy,\quad r>0.\eqno(2.5)$$
Moreover, for any $r>0$, we see that $$|y_1|\leq |x_1|+r,\quad y\in B(x_*,r).\eqno(2.6)$$

Since $a\geq 0$, using (2.5) and (2.6),  for any $x\in \rr^n$ and $r>0$, we obtain
\begin{align*}
\int_{B(x,r)}w(y)dy&=\int_{B(x_*,r)}|y_1|^ady\\
&\leq (|x_1|+r)^a\int_{B(x_*,r)}dy=\omega_nr^n(|x_1|+r)^a.
\end{align*}
This implies (2.2). Hence assertion $\mathrm{(i)}$ holds.

Next we prove assertion $\mathrm{(ii)}$. Since $w(y)=|y|^b$ is monotonically increasing function with respect to the distance from the origin, we have
\begin{align*}
w(B(x,r))&=\int_{B(x,r)}|y|^bdy\geq \int_{B(0,r)}|y|^bdy\\
&=\frac{n\omega_n}{n+b}r^{n+b}
\end{align*}
for all $x\in \rr^n$ and $r>0$. This implies (2.3). On the other hand, for any $r>0$, we see that $$|y|\leq |x|+r,\quad y\in B(x,r).\eqno(2.7)$$
Since $b\geq0$, using $(2.7)$, for any $x\in \rr^n$ and $r>0$, we obtain
\begin{align*}
\int_{B(x,r)}w(y)dy&=\int_{B(x,r)}|y|^bdy\\
&\leq w_n(|x|+r)^br^n.
\end{align*}
This implies (2.4). Hence assertion $\mathrm{(ii)}$ holds, and Lemma 2.1 follows.
\end{proof}

Combing Lemma 2.1 with $\mathrm{(K3)},$ under condition $\mathrm{(A)}$,  we immediately deduce that
\begin{align*}
d&\times\min\left\{t^{-\frac{n}{4}}|x_1|^{-\frac{a}{2}},t^{-\frac{n+a}{4}}\right\}\times\min\left\{t^{-\frac{n}{4}}|y_1|^{-\frac{a}{2}},t^{-\frac{n+a}{4}}\right\}
\mathrm{exp}\left(-\dfrac{|x-y|^2}{dt}\right)\\\tag{2.8}
&\leq \Gamma(x,y,t)\leq Dt^{-\frac{n+a}{2}}\mathrm{exp}\left(-\frac{|x-y|^2}{Dt}\right)
\end{align*}
for $x,y\in \rr^n$ and $t>0$. Here $D$ and $d$ are positive constants depending only on $n$ and $a$. Similarly to (2.8), in the case $\mathrm{(B)}$, we get that
\begin{align*}
d&\times\min\left\{t^{-\frac{n}{4}}|x|^{-\frac{b}{2}},t^{-\frac{n+b}{4}}\right\}\times\min\left\{t^{-\frac{n}{4}}|y|^{-\frac{b}{2}},t^{-\frac{n+b}{4}}\right\}
\mathrm{exp}\left(-\dfrac{|x-y|^2}{dt}\right)\\\tag{2.9}
&\leq \Gamma(x,y,t)\leq Dt^{-\frac{n+b}{2}}\mathrm{exp}\left(-\frac{|x-y|^2}{Dt}\right)
\end{align*}
for $x,y\in \rr^n$ and $t>0$. Here $D$ and $d$ are positive constants depending only on $n$ and $b$.
By (2.8) and $(2.9)$, we get $$\Gamma(x,y,t)\leq Dt^{-\frac{n+\alpha}{2}}$$
for $x,y\in \rr^n$ and $t>0$. Here $\alpha\in\{a,b\}.$ This together with $\mathrm{(K1)}$ implies that
$$\|\Gamma(\cdot,y,t)\|_{L^r(w)}\leq Ct^{-\frac{n+a}{2}\left(1-\frac{1}{r}\right)},\quad \|\Gamma(x,\cdot,t)\|_{L^r(w)}\leq Ct^{-\frac{n+a}{2}\left(1-\frac{1}{r}\right)},\eqno(2.10)$$
for any $1\leq r\leq \infty,$ where $C$ depends only on $n$ and $\alpha\in \{a,b\}$. Moreover, we have the following.
\begin{lemma}
Assume either $\mathrm{(A)}$ or $\mathrm{(B)}$. Let $\alpha$ be such that  $\alpha=a$ for the case $\mathrm{(A)}$ and $\alpha=b$ for the case $\mathrm{(B)}$. Then, for any $1\leq r<\infty$, there exists a positive constant $C$ depending only on $\alpha, r$ and $n$ such that $$\|\Gamma(x,\cdot,t)\|_{L^{r,1}(w)}\leq Ct^{-\frac{n+a}{2}\left(1-\frac{1}{r}\right)},\eqno(2.11)$$
for $x\in \rr^n$ and $t>0.$
\end{lemma}
\begin{proof}
By $(2.8)$ and $(2.9)$, if we can prove (2.11) for the case $\mathrm{(A)}$, then, replacing  $|x_1|$ and $a$ with $|x|$ and $b$, respectively, we can obtain (2.11) for the case $\mathrm{(B)}$. So it suffices to prove (2.11) for the case $\mathrm{(A)}$.

Assume $\mathrm{(A)}$. For any $x\in \rr^n$ and $t>0$, set $$f(y):=Dt^{-\frac{n+a}{2}}e^{-\frac{|x-y|^2}{Dt}}\quad\text{for}~y\in\rr^n, \eqno(2.12)$$
where $D$ is the constant given in (2.8). By (1.5), we get
$$\mu_f(\lambda)=w\left(\{y:|f(y)|>\lambda\}\right)=w\left(B\left(x,\sqrt{-Dt\mathrm{log}\left(D^{-1}t^{\frac{n+a}{2}}\lambda\right)}\right)\right).$$
Furthermore, using Lemma 2.1, we obtain
\begin{equation*}\label{1}
\mu_f(\lambda)    \leq C'\left \{
   \begin{array}{ll}
\left[-Dt\mathrm{log}\left(D^{-1}t^{\frac{n+a}{2}}\lambda\right)\right]^{\frac{n}{2}}|x_1|^a\quad & \text{if} ~0< \sqrt{-Dt\mathrm{log}\left(D^{-1}t^{\frac{n+a}{2}}\lambda\right)}\leq |x_1|,\\\tag{2.13}
\left[-Dt\mathrm{log}\left(D^{-1}t^{\frac{n+a}{2}}\lambda\right)\right]^{\frac{n+a}{2}}& \text{if}~ \sqrt{-Dt\mathrm{log}\left(D^{-1}t^{\frac{n+a}{2}}\lambda\right)}\geq |x_1|,
 \end{array}
 \right.
\end{equation*}
where $C'$ is the constant given in (2.2). By (1.6) and (2.13), we deduce that
\begin{equation*}\label{1}
f^*(s)    \leq\left \{
   \begin{array}{ll}
Dt^{-\frac{n+a}{2}}e^{-\frac{C''}{D}\left(t^{-\frac{n}{2}}|x_1|^{-a}s\right)^\frac{2}{n}}\quad &\text{if}~s\geq |x_1|^{n+a},\\\tag{2.14}
Dt^{-\frac{n+a}{2}}e^{-\frac{C''}{D}\left(t^{-\frac{n+a}{2}}s\right)^{\frac{2}{n+a}}}\quad &\text{if}~ s<|x_1|^{n+a},
 \end{array}
 \right.
\end{equation*}
where $C''$ depends only on $n$ and $a$.
Then, by (1.7) and $(2.14)$ we have
\begin{align*}
\|f\|_{L^{r,1}(w)}&=\int^\infty_0s^{\frac{1}{r}-1}f^*(s)ds\\
&=Dt^{-\frac{n+a}{2}}\int^{|x_1|^{n+a}}_0s^{\frac{1}{r}-1}e^{-\frac{C''}{D}\left(t^{-\frac{n+a}{2}}s\right)^{\frac{2}{n+a}}}ds\\
&\quad+Dt^{-\frac{n+a}{2}}\int^\infty_{|x_1|^{n+a}}s^{\frac{1}{r}-1}e^{-\frac{C''}{D}\left(t^{-\frac{n}{2}}|x_1|^{-a}s\right)^\frac{2}{n}}ds\\
&\leq Dt^{-\frac{n+a}{2}\left(1-\frac{1}{r}\right)}\int^\infty_0\xi^{\frac{1}{r}-1}e^{-\frac{C''}{D}\xi^{\frac{2}{n+a}}}d\xi\\
&\quad+Dt^{-\frac{n+a}{2}}t^{\frac{n}{2r}}|x_1|^{\frac{a}{r}}\int^\infty_{t^{-\frac{n}{2}}|x_1|^n}\xi^{\frac{1}{r}-1}e^{-\frac{C''}{D}\xi^{\frac{2}{n}}}d\xi\\
&\leq Ct^{-\frac{n+a}{2}\left(1-\frac{1}{r}\right)}+Dt^{-\frac{n+a}{2}}t^{\frac{n}{2r}}|x_1|^{\frac{a}{r}}e^{-\frac{C''}{2D}\frac{|x_1|^2}{t}}
\int^\infty_{t^{-\frac{n}{2}}|x_1|^n}\xi^{\frac{1}{r}-1}e^{-\frac{C''}{2D}\xi^{\frac{2}{n}}}d\xi\\\tag{2.15}
&\leq Ct^{-\frac{n+a}{2}\left(1-\frac{1}{r}\right)}+Ct^{-\frac{n+a}{2}}t^{\frac{n}{2r}}|x_1|^{\frac{a}{r}}e^{-\frac{C''}{2D}\frac{|x_1|^2}{t}}
\end{align*}
for any $1\leq r<\infty,$ where $C$ depends only on $a, r$ and $n$. Since $s^{\frac{a}{2r}}e^{-\frac{C''}{2D}s}\leq C$ for all $s\geq 0,$ we see that
\begin{align*}
&t^{-\frac{n+a}{2}}t^{\frac{n}{2r}}|x_1|^{\frac{a}{r}}e^{-\frac{C''}{2D}\frac{|x_1|^2}{t}}\\
&= t^{-\frac{n+a}{2}}t^{\frac{n}{2r}}|x_1|^{\frac{a}{r}}\left(\frac{|x_1|^2}{t}\right)^{-\frac{a}{2r}}
\left(\frac{|x_1|^2}{t}\right)^{\frac{a}{2r}}e^{-\frac{C''}{2D}\frac{|x_1|^2}{t}}\\
&\leq Ct^{-\frac{n+a}{2}\left(1-\frac{1}{r}\right)}.
\end{align*}
This together with $(2.15)$ implies that $$\|f\|_{L^{r,1}(w)}\leq Ct^{-\frac{n+a}{2}\left(1-\frac{1}{r}\right)},\quad r\in [1,\infty).\eqno(2.16)$$
By (2.8) and (2.12) we see that $\Gamma(x,y,t)\leq f(y)$ for $x,y\in \rr^n$ and $t>0$, and it follows from $(2.16)$ that
$$\|\Gamma(x,\cdot,t)\|_{L^{r,1}(w)}\leq \|f\|_{L^{r,1}(w)}\leq Ct^{-\frac{n+a}{2}\left(1-\frac{1}{r}\right)},\quad r\in [1,\infty),$$
for $x\in \rr^n$ and $t>0$. Hence the proof of Lemma 2.2 is complete.
\end{proof}

For any measurable function $\varphi$, under condition either $\mathrm{(A)}$ or $\mathrm{(B)}$, we define $$[S(t)\varphi w](x):=\int_{\rr^n}\Gamma(x,y,t)\varphi(y)w(y)dy\eqno(2.17)$$
for all $x\in \rr^n$ and $t>0.$

We will prove $L^q(w)$--$L^r(w)$ estimate and $L^{q,\infty}(w)$--$L^{r,\infty}(w)$ estimate for $S(t)\varphi w$. We first give $L^q(w)$--$L^r(w)$ estimate.
\begin{lemma}
For any $\varphi\in L^q(w)$ and $1\leq q \leq r\leq \infty,$ it holds that $$\|S(t)\varphi w\|_{L^r(w)}\leq c_1t^{-{\frac{n+a}{2}}\left(\frac{1}{q}-\frac{1}{r}\right)}\|\varphi\|_{L^q(w)},\quad t>0.$$
Here $c_1$ can be taken so that it depends only on $n$ and $\alpha\in\{a,b\}.$
\end{lemma}
\begin{proof}
Fix $t>0.$ Using the H\"{o}lder inequality and (2.10), we get  $$\|S(t)\varphi w\|_\infty\leq Ct^{-\frac{n+a}{2}\cdot\frac{1}{q}}\|\varphi\|_{L^q(w)}$$
for any $1\leq q\leq \infty.$ Here $C$ depends only on $n$ and $\alpha$. Furthermore, by $\mathrm{(K1)}$ we apply the Jensen inequality and the Fubini theorem to obtain
\begin{align*}
\|S(t)\varphi w\|_{L^q(w)}^q&=\int_{\rr^n}\left|\int_{\rr^n}\Gamma(x,y,t)\varphi(y)w(y)dy\right|^qw(x)dx\\
&\leq \int_{\rr^n}\left(\int_{\rr^n}\Gamma(x,y,t)|\varphi(y)|w(y)dy\right)^qw(x)dx\\
&\leq \int_{\rr^n}\left(\int_{\rr^n}\Gamma(x,y,t)|\varphi(y)|^qw(y)dy\right)w(x)dx\\
&=\int_{\rr^n}|\varphi(y)|^q\left(\int_{\rr^n}\Gamma(x,y,t)w(x)dx\right)w(y)dy=\|\varphi\|_{L^q(w)}^q.
\end{align*}
These implies that $$\|S(t)\varphi w\|_{L^r(w)}\leq \|S(t)\varphi w\|_\infty^{\frac{r-q}{r}}\|S(t)\v w\|^{\frac{q}{r}}_{L^q(w)}\leq C^{\frac{r-q}{r}}t^{-\frac{n+a}{2}\left(\frac{1}{q}-\frac{1}{r}\right)}\|\varphi\|_{L^q(w)}.$$
The constant $C^{\frac{r-q}{r}}$ is bounded by the constant depending only on $n$ and $\alpha$ for all $1\leq q\leq r\leq \infty.$ Thus we complete the proof of Lemma 2.3.
\end{proof}

Before proving $L^{q,\infty}(w)$--$L^{r,\infty}(w)$ estimate for $S(t)\varphi w$, we prepare the following lemma.
\begin{lemma}
Let $1<r\leq \infty$. Assume $\varphi\in L^{r,\infty}(w)$. Then there exists a positive constant $C_r$ depending only on $r$ such that
$$\|S(t)\v w\|_{L^{r,\infty}(w)}\leq C_r\|\varphi\|_{L^{r,\infty}(w)}, \quad t>0.\eqno(2.18)$$
The constant $C_r$ is bounded in $1+\varepsilon< r<\infty$ for any fixed $\varepsilon>0$ and $C_r\to \infty $ as $r\to 1.$
\end{lemma}
\begin{proof}
In case of $r=\infty,$ since $L^{\infty,\infty}(w)=L^\infty$, by $\mathrm{(K1)}$, (2.18) holds true with the constant $C_r=1.$ We consider the case where $1<r<\infty.$
Let $M$ be a positive real number to be chosen later. Set $$M_-:=\{x:|\varphi|\leq M\}\quad \text{and}\quad M_+:=\{x:|\varphi|>M\}.$$
Furthermore, we define $$\varphi_1:=\v\chi_{ M_-},\quad \v_2:=\v\chi_{ M_+},$$
where $\chi_E$ is the characteristic function of $E$. By the fundamental properties of $\mu_f(\lambda)$ we have
\begin{equation*}\label{1}
\mu_{\varphi_1}(\lambda)=\left \{
   \begin{array}{ll}
0\qquad \qquad~~ &\text{if}~\lambda \geq M,\\\tag{2.19}
\mu_{\v}(\lambda)-\mu_{\v}(M)&\text{if}~\lambda<M,
 \end{array}
 \right.
\end{equation*}
and
\begin{equation*}\label{1}
\mu_{\varphi_2}(\lambda)=\left \{
   \begin{array}{ll}
\mu_{\v}(\lambda)\qquad \qquad~~ &\text{if}~\lambda \geq M,\\\tag{2.20}
\mu_{\v}(M)&\text{if}~\lambda<M.
 \end{array}
 \right.
\end{equation*}
Since $\{x: S(t)\v w>\lambda\}\subset \{x: S(t)\v_1w>\lambda/2\}\cup \{x: S(t)\v_2w>\lambda/2\}$ for $\lambda >0,$ we get
$$\mu_{S(t)\v w}(\lambda)\leq \mu_{S(t)\v_1w}(\lambda/2)+\mu_{S(t)\v_2w}(\lambda/2)\eqno(2.21)$$
for all $t>0.$
Then, by (1.9) and $(2.20)$ we obtain
\begin{align*}
\|\v_2\|_{L^1(w)}&=\int^\infty_0\mu_{\v_2}(\lambda)d\lambda\\
&=\int^M_0\mu_{\v}(M)d\lambda+\int^\infty_M\mu_{\v}(\lambda)d\lambda\\
&\leq M\mu_{\v}(M)+\int^\infty_M\lambda^{-r}\|\v\|^r_{L^{r,\infty}(w)}d\lambda\\
&\leq M^{1-r}\|\v\|^r_{L^{r,\infty}(w)}+\frac{1}{r-1}M^{1-r}\|\v\|^r_{L^{r,\infty}(w)}\\
&=\frac{r}{r-1}M^{1-r}\|\v\|^r_{L^{r,\infty}(w)}.
\end{align*}
By $\mathrm{(K1)}$ we apply the Fubini theorem to deduce
$$\|S(t)\v_2 w\|_{L^1(w)}\leq \|\v_2\|_{L^1(w)}\leq \frac{r}{r-1}M^{1-r}\|\v\|^r_{L^{r,\infty}(w)}\eqno(2.22)$$
for all $t>0$. On the other hand, using the H\"{o}lder inequality and $\mathrm{(K1)}$, we have that
$$\left|[S(t)\v_1 w](x)\right|\leq \|\Gamma(x,\cdot,t)\|_{L^1(w)}\|\v_1\|_\infty\leq M\eqno(2.23)$$
for all $x\in \rr^n$  and $t>0.$

Let $\nu>0$ and fix it. Taking the constant $M$ as $M=\nu/2$, by (2.23) we get $$\mu_{S(t)\v_1w}(\nu/2)=0,\quad t>0.$$
Then, applying the Chebyshev inequality with  (2.21) and (2.22) $$\mu_{S(t)\v w}(\nu)\leq \mu_{S(t)\v_2w}(\nu/2)\leq \frac{2}{\nu}\|S(t)\v_2 w\|_{L^1(w)}\leq \dfrac{2^rr}{r-1}v^{-r}\|\v\|_{L^{r,\infty}(w)}^r$$
for all $t>0$. Since the above inequality holds all $\nu>0$, by (1.9) we obtain (2.18) with $C_r=2(r/(r-1))^{\frac{1}{r}}$, and the constant $C_r$ is bounded as $r\to \infty$. Thus the proof of Lemma 2.4 is complete.
\end{proof}

Next we give  $L^{q,\infty}(w)$--$L^{r,\infty}(w)$ estimate for $S(t)\varphi w$.
\begin{lemma}
For any $\v\in L^{q,\infty}(w)$ with $1<q\leq \infty$ and $q\leq r\leq \infty$, it holds that
$$\|S(t)\v w\|_{L^{r,\infty}(w)}\leq c_2t^{-\frac{n+a}{2}\left(\frac{1}{q}-\frac{1}{r}\right)}\|\v\|_{L^{q,\infty}(w)},\quad t>0.$$
Here $c_2$ can be taken so that it depends only on $q$, $n$ and $\alpha\in\{a,b\}$. In particular, the constant $c_2$ is bounded in $q\in (1+\varepsilon,\infty)$ for any fixed $\varepsilon>0$ and $c_2\to \infty$ as $q\to 1.$
\end{lemma}
\begin{proof}
For any $1<q\leq \infty,$ it follows from (1.11) and (2.11) that
\begin{align*}
\|S(t)\v w\|_\infty&\leq \sup_{x\in \rr^n}\|\Gamma(x,\cdot,t)\varphi\|_{L^1(w)}\\\tag{2.24}
&\leq \|\Gamma(x,\cdot,t)\|_{L^{\frac{q}{q-1},1}(w)}\|\v\|_{L^{q,\infty}(w)}\\
&\leq Ct^{-\frac{n+a}{2}\cdot\frac{1}{q}}\|\v\|_{L^{q,\infty}(w)}
\end{align*}
for all $t>0$, where $C$ depends only on $q, n$ and $\alpha$. Hence, Combining (1.10), (2.18) and $(2.24)$, we deduce
$$\|S(t)\v w\|_{L^{r,\infty}(w)}\leq \|S(t)\v w\|_{L^{q,\infty}(w)}^{1-\theta}\|S(t)\v w\|_\infty^\theta\leq C^{1-\theta}_qC^\theta t^{-\frac{n+a}{2}\cdot\frac{\theta}{q}}\|\v\|_{L^{q,\infty}(w)},$$
where $\theta=1-q/r\in [0,1]$. Since $\theta/q=1/q-1/r$, we complete the proof of Lemma 2.5.
\end{proof}

Furthermore, by (2.8) and $(2.9)$ we have the following lemma.
\begin{lemma}
Assume same conditions as in Lemma $2.2.$ Let $\varphi\in L^\infty$ be a measurable function such that $\v\geq 0 $ in $\rr^n$. Then there exists a positive constant $C$ depending only on $n$ and $\alpha\in\{a,b\}$ such that $$[S(t)\v w](x)\geq Ct^{-\frac{n+\alpha}{2}}\int _{|y|\leq \sqrt{t}}\v (y)w(y)dy$$
for $|x|\leq \sqrt{t}$ and $t>0$.
\end{lemma}
\begin{proof}
Since $|x-y|^2\leq 2(|x|^2+|y|^2)$, by (2.8) and (2.9) we can find a positive constant $C$ depending only on $n$ and $\alpha$ such that $$\Gamma(x,y,t)\geq Ct^{-\frac{n+\alpha}{2}}$$
for $|x|, |y|\leq \sqrt{t}$ and $t>0.$

Then, by (2.17) we get
$$[S(t)\v w](x)\geq \int_{|y|\leq \sqrt{t}}\Gamma(x,y,t)\v(y)w(y)dy\geq Ct^{-\frac{n+\alpha}{2}}\int _{|y|\leq \sqrt{t}}\v(y)w(y)dy$$
for all $|x|\leq \sqrt{t}$ and $t>0.$ Thus we complete the proof of Lemma 2.6.
\end{proof}
\section{Proof of Theorem 1.1}
In what follows,  we assume either $\mathrm{(A)}$ or $\mathrm{(B)}$. Let $\alpha$ be such that $\alpha=a$ for the case $\mathrm{(A)}$
and $\alpha=b$ for the case $\mathrm{(B)}$. In this section, we prove Theorem 1.1, which means that problem (1.3) has no positive global solutions in the case $1<p\leq p_*(\alpha)$. The proof of Theorem 1.1 is based on the arguments of \cite[Theorem 5]{WF1} and \cite[Theorem 1]{WF2} (see also \cite[Theorem 1.1]{FI} and \cite[Lemma 3.1]{FU}).

Before proving Theorem 1.1, we first give the following lemma.
\begin{lemma}
Let $u$ be a solution of $(1.3)$ in $\rr^n\times (0,T)$ with $0<T\leq \infty.$ Then there exists a constant $C^*$ depending only on $p$ such that
$$t^{\frac{1}{p-1}}\|S(t)u_0w\|_\infty\leq C^*\eqno(3.1)$$
for any $t\in [0,T).$
\end{lemma}
\begin{proof}
Since it follows from (2.8) and (2.9) that the fundamental solution $\Gamma$ is positive for $x, y\in \rr^n$ and $t>0,$ by (1.4) and $(2.17)$ we have
$$[S(t)u_0 w](x)\leq u(x,t)<\infty\eqno(3.2)$$
for almost all $x\in \rr^n$ and all $t\in (0,T).$ This together with (1.4) again implies $$u(x,t)\geq \int^t_0\int_{\rr^n}\Gamma(x,y,t-s)([S(s)u_0 w](y))^pw(y)dyds\eqno(3.3)$$
for almost all $x\in \rr^n$ and all $t\in (0,T).$ Then, by $\mathrm{(K1)}$ and $\mathrm{(K2)}$ we apply the Jensen inequality to (3.3), we deduce
\begin{align*}
u(x,t)&\geq \int^t_0\left(\int_{\rr^n}\Gamma(x,y,t-s)[S(s)u_0 w](y)w(y)dy\right)^pds\\\tag{3.4}
&=\int^t_0\left(\int_{\rr^n}\Gamma(x,y,t-s)\int_{\rr^n}\Gamma(y,z,s)u_0(z)w(z)dzw(y)dy\right)^pds\\
&=\int^t_0\left[\int_{\rr^n}\left(\int_{\rr^n}\Gamma(x,y,t-s)\Gamma(y,z,s)w(y)dy\right)u_0(z)w(z)dz\right]^pds\\
&=\int^t_0\left(\int_{\rr^n}\Gamma(x,z,t)u_0(z)w(z)dz\right)^pds=t([S(t)u_0w](x))^p
\end{align*}
for almost all $x\in \rr^n$ and all $t\in (0,T).$ We repeat the above argument with (3.2) replaced by (3.4), and get
\begin{align*}
u(x,t)&\geq \int^t_0\int_{\rr^n}\Gamma(x,y,t-s)\left(s([S(s)u_0w](y))^p\right)^pw(y)dyds\\
&\geq \int^t_0s^p\left(\int_{\rr^n} \Gamma(x,y,t-s)[S(s)u_0w](y)w(y)dy\right)^{p^2}ds\\
&=\dfrac{1}{p+1}~t^{p+1}([S(t)u_0w](x))^{p^2}
\end{align*}
for almost all $x\in \rr^n$ and all $t\in (0,T).$ Repeating the above argument, for any $k=2,3,\ldots,$
it holds that $$u(x,t)\geq A_k~t^{\frac{p^k-1}{p-1}}([S(t)u_0w](x))^{p^k}\eqno(3.5)$$
for almost all $x\in \rr^n$ and all $t\in (0,T),$ where
\begin{align*}
A_k:=&\left(\dfrac{1}{p+1}\right)^{p^{k-2}}\left(\dfrac{1}{(p+1)p+1}\right)^{p^{k-3}}\cdots\left(\dfrac{1}{(1+p+\cdots+p^{k-2})p+1}\right)\\
=&\prod^{k-1}_{j=1}\left(\dfrac{p-1}{p^{j+1}-1}\right)^{p^{k-j-1}}.
\end{align*}
Hence, by (3.5) we get $$t^{\frac{1}{p-1}(1-p^{-k})}[S(t)u_0w](x)\leq u(x,t)^{p^{-k}}\left(\prod^{k-1}_{j=1}\left(\dfrac{p-1}{p^{j+1}-1}\right)^{p^{k-j-1}}\right)^{-p^{-k}}\eqno(3.6)$$
for almost all $x\in \rr^n$ and all $t\in (0,T).$
In addition, we have
\begin{align*}
\mathrm{log}\left(\prod^\infty_{j=1}\left(\dfrac{p^{j+1}-1}{p-1}\right)^{p^{-j-1}}\right)&=\sum^{\infty}_{j=1}p^{-j-1}\mathrm{log}\left(\dfrac{p^{j+1}-1}{p-1}\right)\\\tag{3.7}
&\leq \sum^\infty_{j=1}p^{-j-1}\mathrm{log}\left((j+1)p^j\right)\\
&\leq (1+\mathrm{log}p)\sum^\infty_{j=2}jp^{-j}<\infty.
\end{align*}
Then, by (3.6) and (3.7) we can find a constant $C^*$ depending only on $p$ such that $$t^{\frac{1}{p-1}}[S(t)u_0w](x)\leq C^*<\infty$$
for almost all $x\in \rr^n$ and all $t\in (0,T).$ This implies (3.1), and Lemma 3.1 follows.
\end{proof}

Now we prove Theorem 1.1 by using Lemma 3.1.\\
\\
{\bf{Proof of Theorem 1.1}}\quad The proof is by contradiction. Let $u$ be a positive global-in-time solution of (1.3). Since $u(\cdot, 1)$ is a positive measurable function in $\rr^n$, we can find a non-trivial measurable function $U_1\in L^\infty$ such that $\supp U_1\subset B(0,1)$ and $0\leq U_1(x)\leq u(x,1)$ for almost all $x\in \rr^n$. Then, using Lemma 2.6, we have $$[S(t)U_1 w](x)\geq CM t^{-\frac{n+\alpha}{2}},\quad M:=\int_{B(0,1)}U_1(x)w(x)dx,\eqno(3.8)$$
for all $|x|\leq \sqrt{t}$ and $t\geq 1.$ Furthermore, by (1.4) and  (2.17) we see that $$u(x,t+1)\geq [S(t)u(1)w](x)\geq [S(t)U_1w](x)\eqno(3.9)$$
for almost all $x\in \rr^n$ and all $t>0.$

We first consider the case $1<p<p_*(\alpha)$. By (3.8) and (3.9) we get $$[S(t)u(1)w](x)\geq CM t^{-\frac{n+\alpha}{2}}\eqno(3.10)$$
for all $|x|\leq \sqrt{t}$ and $t\geq 1$. It follows from  $1<p<p_*(\alpha)$ with (3.10) that $$t^{\frac{1}{p-1}}\|S(t)u(1)w\|_\infty\to\infty\quad \text{as}\quad t\to \infty,$$
which contradicts (3.1). This means that problem (1.3) has no positive global-in-time positive solutions.

Next we consider the case $p=p_*(\alpha)$. Since $t+1-s\leq t$ and $s\leq t+1-s$ for $1\leq s\leq t/2$, by (2.8) and (2.9) we have $$\int_{|x|\leq \sqrt{t}}\Gamma(x,y,t)w(x)dx\geq Ct^{-\frac{n+\alpha}{2}}\int_{|x|\leq \sqrt{t}}w(x)dx\geq C\eqno(3.11)$$
for all $|y|\leq \sqrt{t}$. It follows from (1.4), (3.9), (3.10) and (3.11) that
\begin{align*}
&\int_{|x|\leq \sqrt{t}}u(x,t+1)w(x)dx\\
&\geq \int_{|x|\leq \sqrt{t}}\int^{\frac{t}{2}}_1\int_{|y|\leq \sqrt{t+1-s}}\Gamma(x,y,t+1-s)u(y,s)^pw(y)dydsw(x)dx\\\tag{3.12}
&\geq \int^{\frac{t}{2}}_1\int_{|y|\leq \sqrt{t+1-s}}\left(\int_{|x|\leq \sqrt{t+1-s}}\Gamma(x,y,t+1-s)w(x)dx\right)u(y,s)^pw(y)dyds\\
&\geq C\int^{\frac{t}{2}}_1\int_{|y|\leq \sqrt{t+1-s}}u(y,s)^{p-1}u(y,s)w(y)dyds\\
&\geq CM^p\int^{\frac{t}{2}}_1\left(s^{-\frac{n+\alpha}{2}}\right)^{p-1}\left(\int_{|y|\leq \sqrt{s}}s^{-\frac{n+\alpha}{2}}w(y)dy\right)ds\\
&\geq CM^p\int^{\frac{t}{2}}_1s^{-\frac{n+\alpha}{2}(p-1)}ds\geq CM^p\mathrm{log}t,\quad t>3.\\
\end{align*}
Let $m$ be a sufficiently large positive constant. By (3.12) we can find $T>0$ such that the function $U_2$ defined by $U_2:=u(\cdot,T)\in L^\infty$ satisfies
$$\int_{|x|\leq \sqrt{T}}U_2(x)w(x)dx\geq m.$$
Similarly to (3.8) and (3.9), we have $$u(x,t+T)\geq [S(t)U_2w](x)\geq Cm t^{-\frac{n+\alpha}{2}}$$
for almost all $x\in \rr^n$ and all $t>0$. This implies that $$t^{\frac{n+\alpha}{2}}\|S(t)U_2w\|_{\infty}\geq Cm,\quad t>1.\eqno(3.13)$$
Let $v$ be a solution of (1.3) with initial $U_2$. Then, since $u$ is a positive global-in-time solution of (1.3), $v$ is also a global-in-time solution of (1.3). Thus we can apply Lemma 3.1 to the solution $v$, and obtain (3.1) replacing $u_0$ with $U_2$. By the arbitrariness of $m$, this contradicts $(3.13)$ and we see that problem (1.3) has no positive global-in-time solutions for the case $p=p_*(\alpha)$. Hence the proof of Theorem 1.1 is complete. \qed

\section{Proof of Theorem 1.2}
In this section, we prove the uniqueness and local existence of solutions of (1.3) and then obtain Theorem 1.2. We first give the uniqueness of solutions of (1.3); see also  \cite[Lemma 3.1]{FI1}.
\begin{lemma}
Let $i=1,2$, $\tau>0$, and $u_i$ be a solution of $(1.3)$ in $\rr^n\times (0,\tau)$ with $u_0=u_{0,i}\in L^\infty.$ Then, for any $\sigma\in (0,\tau)$, there exists a constant $C$ such that $$\sup_{0<t\leq \sigma}\|u_1(t)-u_2(t)\|_\infty\leq C\|u_{0,1}-u_{0,2}\|_\infty,\eqno(4.1)$$
where the constant $C$ depends on $\|u_1\|_{L^\infty(0,\sigma : L^\infty)}$ and $\|u_2\|_{L^\infty(0,\sigma : L^\infty)}$.
\end{lemma}
\begin{remark}
Let $\tau>0$ and $u$ be a  solution of $(1.3)$ in $\rr^n\times (0,\tau].$ If $\|u\|_{L^\infty(0,\tau : L^\infty)}$ is bounded, then we can take a constant in $(4.1)$ uniformly with respect to $\sigma$. Hence, if $u$ is a global-in-time bounded solution of $(1.3)$, then, apply Lemma $4.1$, we see that $u$ is a unique solution of $(1.3)$.
\end{remark}
\begin{proof}
Let $\sigma\in (0,\tau).$ Set $v=u_1-u_2$. Then we have $$\|v\|_{L^\infty(0,\sigma : L^\infty)}\leq \|u_1\|_{L^\infty(0,\sigma : L^\infty)}+ \|u_2\|_{L^\infty(0,\sigma : L^\infty)}<\infty.$$
It follows from (1.4) and $\mathrm{(K1)}$ that
\begin{align*}
|v(x,\tilde{t})|&\leq \|v(t)\|_\infty+\int^{\tilde{t}}_t\int_{\rr^n}\Gamma(x,y,\tilde{t}-s)|u_1(y,s)|^p-u_2(y,s)^p|w(y)dyds\\
&\leq \|v(t)\|_\infty+C_1\int^{\tilde{t}}_t\int_{\rr^n}\Gamma(x,y,\tilde{t}-s)|v(y,s)|w(y)dyds\\
&\leq \|v(t)\|_\infty+C_1\sup_{t<\tau\leq \tilde{t}}\|v(\tau)\|_\infty(\tilde{t}-t)
\end{align*}
for almost all $x\in \rr^n$ and all $0\leq t<\tilde{t}\leq \sigma$, where $C_1$ is a positive constant depending only on $p$, $\|u_1\|_{L^\infty(0,\sigma : L^\infty)}$ and
$\|u_2\|_{L^\infty(0,\sigma : L^\infty)}$. This implies that $$\sup_{t<\tau\leq \tilde{t}}\|v(\tau)\|_\infty\leq \|v(t)\|_\infty+C_1\sup_{t<\tau\leq \tilde{t}}\|v(\tau)\|_\infty(\tilde{t}-t)\eqno(4.2)$$
for all $0\leq t<\tilde{t}\leq \sigma.$

Let $\varepsilon$ be a sufficiently small positive constant such that $C_1\varepsilon\leq 1/2$ and $\varepsilon<\sigma$. Then, using (4.2), we obtain $$\sup_{t<\tau\leq t+\varepsilon}\|v(\tau)\|_\infty\leq 2\|v(t)\|_\infty$$
for all $t\in [0,\sigma-\varepsilon]$. Hence there exists a constant $C_2$ such that $$\sup_{0<t\leq \sigma}\|v(t)\|_\infty\leq C_2\|v(0)\|_\infty,$$
and we have inequality (4.1). Thus the proof of Lemma 4.1 is complete.
\end{proof}

Next we prove local existence of solutions of (1.3). For any nonnegative function $u_0\in L^\infty$, we define $\{u_n\}$ inductively by
\begin{align*}
u_1(x,t)&:=\int_{\rr^n}\Gamma(x,y,t)u_0(y)w(y)dy\\
u_{n+1}(x,t)&:=u_1(x,t)+\int^t_0\int_{\rr^n}\Gamma(x,y,t-s)u_n(y,s)^pw(y)dyds,\quad n=1,2,\ldots,\tag{4.3}
\end{align*}
for almost all $x\in \rr^n$ and all $t>0.$ Then we can prove inductively that $$0\leq u_n(x,t)\leq u_{n+1}(x,t)\eqno(4.4)$$
for almost all $x\in \rr^n, t>0$ and all $n\in \{1,2,\ldots\}$. Indeed, we clearly obtain $u_2\geq u_1$ since $\Gamma, w$ and $u_1$ are nonnegative functions, and if there exists a number $k\in \{1,2,\ldots\}$ such that $u_k(x,t)\leq u_{k+1}(x,t)$ for almost all $x\in \rr^n$ and all $t>0$, then
\begin{align*}
u_{k+2}(x,t)&=u_1(x,t)+\int^t_0\int_{\rr^n}\Gamma(x,y,t-s)u_{k+1}(y,s)^pw(y)dyds\\
&\geq u_1(x,t)+\int^t_0\int_{\rr^n}\Gamma(x,y,t-s)u_{k}(y,s)^pw(y)dyds=u_{k+1}(x,t)
\end{align*}
for $x\in \rr^n$ and $t>0$. This implies that $(4.4)$ holds true for all $n\in \{1,2,\ldots\}$. Hence we can see that the limit function $$u_*(x,t):=\lim_{n\to \infty}u_n(x,t)\in [0,\infty]\eqno(4.5)$$
can be defined for all almost $x\in \rr^n$ and all $t>0$. Furthermore, by Lemma 2.3 and Lemma 2.5 we can set a constant $c_{**}=\max\{c_1,c_2\}$ such that
\begin{align*}
&\sup_{0<t<\infty}\|u_1(t)\|_\infty\leq c_{**}\|u_0\|_\infty,\\\tag{4.6}
&\sup_{0<t<\infty}t^{\frac{n+\alpha}{2}\left(\frac{1}{r_*}-\frac{1}{q}\right)}\|u_1(t)\|_{L^{q,\infty}(w)}\leq c_{**}\|u_0\|_{L^{r_*,\infty}(w)},
\end{align*}
for any $q\in [r_*,\infty]$ if $u_0\in L^{r_*,\infty}(w)\cap L^\infty,$ where $c_1$ and $c_2$ are constants given in Lemma 2.3 and Lemma 2.5, respectively.
Then we have the following lemma, which implies the local existence of solutions of (1.3). (See also \cite[Lemma 3.2]{FI1} and \cite[Lemma 3.1]{IK1}.)
\begin{lemma}
Let $u_0\in L^\infty$. Then there exists  a positive constant $T$ such that the problem $(1.3)$ has a unique solution $u$ of $(1.3)$ in $\rr^n\times(0,T)$ satisfying
$$\sup_{0<t<T}\|u(t)\|_\infty\leq 2c_{**}\|u_0\|_\infty.\eqno(4.7)$$
Here $c_{**}$ is the constant given in $(4.6).$
\end{lemma}
\begin{proof}
Let $T$ be a sufficiently small positive constant to be chosen later. By induction we prove that $$\sup_{0<t<T}\|u_n(t)\|_\infty\leq 2c_{**}\|u_0\|_\infty\eqno(4.8)$$
for all $n=1,2,\ldots$. By (4.6) we get $(4.8)$ for $n=1.$ Assume that $(4.8)$ holds for some $n=n_*\in\{1,2,\ldots\}$, that is, $$\sup_{0<t<T}\|u_{n_*}(t)\|_\infty\leq 2c_{**}\|u_0\|_\infty.$$
Then, by $(4.3)$ and Lemma 2.3 we obtain
\begin{align*}
\|u_{n_*+1}(t)\|_\infty&\leq \|u_1(t)\|_\infty+\int^t_0\|S(t-s)u_{n_*}(s)^p w\|_\infty ds\\\tag{4.9}
&\leq c_{**}\|u_0\|_\infty+C_1\int^t_0\|u_{n_*}(s)\|^p_\infty ds\\
&\leq c_{**}\|u_0\|_\infty+C_1T(2 c_{**}\|u_0\|_\infty)^p
\end{align*}
for all $t\in (0,T)$, where $C_1$ is a constant independent of $n_*$ and $T$. Let $T$ be a sufficiently small constant such that $$C_1T2^p(c_{**}\|u_0\|_\infty)^{p-1}\leq 1.$$
Then by (4.9) we get $(4.8)$ for $n=n_*+1$. Hence (4.8) holds for all $n=1,2,\ldots$. By (4.4), (4.5) and $(4.8)$ we have that the limit function $u_*$ satisfies (1.4) and $$\sup_{0<t<T}\|u_*(t)\|_\infty\leq 2c_{**}\|u_0\|_\infty.$$
This together with Lemma 4.1 implies that the function $u=u_*$ is a solution of (1.3) in $\rr^n\times(0,T)$. Thus we complete the proof of Lemma 4.2.
\end{proof}

Now we are ready to prove Theorem 1.2.\\
\\
{\bf{Proof of the assertion $\mathrm{\bf{({i})}}$} of Theorem 1.2} Assume (1.12). Let $\delta$ be a sufficiently small positive constant.
Assume (1.13). By induction we prove
\begin{align*}
\|u_n(t)\|_{L^{r_*,\infty}(w)}&\leq 2c_{**}\delta,\\\tag{4.10}
\|u_n(t)\|_\infty&\leq 2c_{**}\delta t^{-\frac{n+\alpha}{2r_*}},
\end{align*}
for all $t>0.$ By (4.6) we get $(4.10)$ for $n=1$. Assume that $(4.10)$ holds for some $n=n_*\in\{1,2,\ldots\},$ that is,
\begin{align*}
\|u_{n_*}(t)\|_{L^{r_*,\infty}(w)}&\leq 2c_{**}\delta,\\\tag{4.11}
\|u_{n_*}(t)\|_\infty&\leq 2c_{**}\delta t^{-\frac{n+\alpha}{2r_*}},
\end{align*}
for all $t>0.$ These together with $(1.10)$ imply that $$\|u_{n_*}(t)\|_{L^{q,\infty}(w)}\leq \|u_{n_*}(t)\|_{L^{r_*,\infty}(w)}^{\frac{r_*}{q}}\|u_{n_*}(t)\|_\infty^{1-\frac{r_*}{q}}\leq 2c_{**}\delta t^{-\frac{n+\alpha}{2}\left(\frac{1}{r_*}-\frac{1}{q}\right)}\eqno(4.12)$$
for all $t>0$ and $r_*<q<\infty.$ Since $r_*=\frac{n+\alpha}{2}(p-1)$, by (4.11) we get $$\|u_{n_*}(t)^p\|_{L^\infty}=\|u_{n_*}(t)\|^p_\infty\leq \left(2c_{**}\delta t^{-\frac{n+\alpha}{2r_*}}\right)^p=(2c_{**}\delta)^pt^{-\frac{p}{p-1}}\eqno(4.13)$$
for all $t>0.$ Similarly, for any $\eta>1$ with $\eta\leq r_*<\eta p$, by (4.12) we obtain $$\|u_{n_*}(t)^p\|_{L^{\eta,\infty}(w)}=\|u_{n_*}(t)\|^p_{L^{\eta p,\infty}(w)}\leq \left(2c_{**}\delta t^{-\frac{n+\alpha}{2}\left(\frac{1}{r_*}-\frac{1}{\eta p}\right)}\right)^p\leq C\delta^p t^{\frac{n+\alpha}{2\eta}-\frac{p}{p-1}}\eqno(4.14)$$
for all $t>0.$ Hence, by Lemma 2.3, Lemma 2.5, (4.13) and (4.14) with $\eta=r_*$ we deduce
\begin{align*}
\left\|\int^t_{\frac{t}{2}}S(t-s)u_{n_*}(s)^pwds\right\|_\infty&\leq \int^t_{\frac{t}{2}}\|S(t-s)u_{n_*}(s)^pw\|_\infty ds\\\tag{4.15}
&\leq \int^t_{\frac{t}{2}}\|u_{n_*}(s)^p\|_\infty ds\\
&\leq C\delta^p\int^t_{\frac{t}{2}}s^{-\frac{p}{p-1}}ds\leq C\delta^p t^{-\frac{1}{p-1}}
\end{align*}
and
\begin{align*}
\left\|\int^t_{\frac{t}{2}}S(t-s)u_{n_*}(s)^pwds\right\|_{L^{r_*,\infty}(w)}&\leq \int^t_{\frac{t}{2}}\|S(t-s)u_{n_*}(s)^pw\|_{L^{r_*,\infty}(w)} ds\\\tag{4.16}
&\leq \int^t_{\frac{t}{2}}\|u_{n_*}(s)^p\|_{L^{r_*,\infty}(w)}ds\\
&\leq C\delta^p\int^t_{\frac{t}{2}}s^{-1}ds\leq C\delta^p
\end{align*}
for all $t>0$. On the other hand, by Lemma 2.5, (4.13) and  (4.14) with $\eta<r_*$ we deduce
\begin{align*}
\left\|\int^{\frac{t}{2}}_0 S(t-s)u_{n_*}(s)^pwds\right\|_\infty&\leq \int_0^{\frac{t}{2}}\|S(t-s)u_{n_*}(s)^pw\|_\infty ds\\\tag{4.17}
&\leq C\int _0^{\frac{t}{2}}(t-s)^{-\frac{n+\alpha}{2\eta}}\|u_{n_*}(s)^p\|_{L^{\eta,\infty}(w)} ds\\
&\leq C\delta^pt^{-\frac{n+\alpha}{2\eta}}\int_0^{\frac{t}{2}}s^{\frac{n+\alpha}{2\eta}-\frac{p}{p-1}}ds\leq C\delta^p t^{-\frac{1}{p-1}}
\end{align*}
and
\begin{align*}
\left\|\int^{\frac{t}{2}}_0 S(t-s)u_{n_*}(s)^pwds\right\|_{L^{r_*,\infty}(w)}&\leq \int_0^{\frac{t}{2}}\|S(t-s)u_{n_*}(s)^pw\|_{L^{r_*,\infty}(w)} ds\\\tag{4.18}
&\leq C\int _0^{\frac{t}{2}}(t-s)^{-\frac{n+\alpha}{2}\left(\frac{1}{\eta}-\frac{1}{r_*}\right)}\|u_{n_*}(s)^p\|_{L^{\eta,\infty}(w)} ds\\
&\leq C\delta^pt^{-\frac{n+\alpha}{2}\left(\frac{1}{\eta}-\frac{1}{r_*}\right)}\int^{\frac{t}{2}}_0 s^{\frac{n+\alpha}{2\eta}-\frac{p}{p-1}}ds\leq C\delta^p
\end{align*}
for all $t>0.$ Then, taking a sufficiently small $\delta$ if necessary, it follows from  (4.6), (4.15), (4.16), (4.17), (4.18) and $r_*=\frac{n+\alpha}{2}(p-1)$ that
\begin{align*}
t^{\frac{1}{p-1}}\|u_{n_*+1}(t)\|_\infty&\leq c_{**} \delta+C_1\delta^p\leq 2c_{**}\delta\\
\|u_{n_*+1}(t)\|_{L^{r_*,\infty}(w)}&\leq  c_{**} \delta+C_1\delta^p\leq 2c_{**}\delta
\end{align*}
for all $t>0$. Here $C_1$ is a constant independent of $n_*$ and $\delta$. Therefore we obtain (4.10) for $n=n_*+1$. Hence $(4.10)$ holds for all $n=1,2,\ldots$.
Hence, applying a similar argument as in the proof of Lemma 4.2, by (4.10) we have that there exists a unique  global-in-time solution $u$ of (1.3) such that
$$\|u(t)\|_{L^{r_*,\infty}(w)}\leq 2c_{**}\delta,\quad \|u(t)\|_\infty\leq 2c_{**}\delta t^{-\frac{1}{p-1}}$$
for all $t>0.$ This together with (4.7) implies that $$\|u(t)\|_\infty\leq C(1+t)^{-\frac{1}{p-1}}$$
for all $t>0$. Furthermore, using the interpolation inequality (1.10) and $r_*=\frac{n+\alpha}{2}(p-1)$, we obtain $$\|u(t)\|_{L^{q,\infty}(w)}\leq C(1+t)^{-\frac{1}{p-1}\left(1-\frac{r_*}{q}\right)}=C(1+t)^{-\frac{n+\alpha}{2}\left(\frac{1}{r_*}-\frac{1}{q}\right)},\quad r_*\leq q\leq \infty,$$
for all $t>0.$ Thus we get (1.14), and the proof of the assertion $\mathrm{(i)}$ of Theorem 1.2 is complete.\qed\\
\\
{\bf{Proof of the assertion $\mathrm{(\bf{ii})}$ of Theorem 1.2}}\quad Assume (1.12). Let $\delta$ be a sufficiently small constant and assume (1.15). Then, it follows from (1.10) and the assertion of $\mathrm{(i)}$ of Theorem 1.2 that there exists a unique global-in-time solution $u$ of (1.3) satisfying (1.14).

We next prove the existence of global-in-time solution of (1.3) satisfying (1.16). For $r=r_*$, it follows from a similar argument as in the proof of the assertion $\mathrm{(i)}$ of Theorem 1.2. So we assume $1\leq r<r_*.$ Set $$u_{0,\lambda}(x):=\lambda^\beta u_0(\lambda x),\quad u_{n,\lambda}(x,t):=\lambda^\beta u_n(\lambda x, \lambda^{2}t),$$
where $\beta=(n+\alpha)/r^*$ and $\lambda$ is a positive constant such that $$\|u_{0,\lambda}\|_{L^r(w)}=\|u_{0,\lambda}\|_\infty.$$
Since $$\|u_{0,\lambda}\|^{\frac{r}{r_*}}_{L^r(w)}\|u_{0,\lambda}\|^{1-\frac{r}{r_*}}_\infty=\|u_0\|^{\frac{r}{r_*}}_{L^r(w)}\|u_0\|^{1-\frac{r}{r_*}}_\infty,$$ it follows from (1.15) that $$\|u_{0,\lambda}\|_{L^r(w)}=\|u_{0,\lambda}\|_\infty<\delta.\eqno (4.19)$$
On the other hand, $u_{n,\lambda}$ satisfies $$u_{n,\lambda}(t)=S(t-\tau)u_{0,\lambda}(\tau)w+\int^t_\tau S(t-s)u_{n-1,\lambda}(s)^pwds,\eqno(4.20)$$
for all $t>\tau\geq 0.$ Moreover, by Lemma 2.3, (4.6) and (4.19) we can find a constant $C_{**}$ independent of $\delta$, $q$ and $r$, such that
$$\|S(t)u_{0,\lambda}w\|_{L^q(w)}\leq C_{**}\delta(1+t)^{-\frac{n+\alpha}{2}\left(\frac{1}{r}-\frac{1}{q}\right)},\quad t>0,\eqno(4.21)$$
for any $q\in[r,\infty].$

By induction we prove that $$\|u_{n,\lambda}(t)\|_{L^q(w)}\leq 2C_{**}\delta,\quad 0<t\leq 2,\eqno(4.22)$$
for any $q\in[r,\infty]$ and $n=1,2,\ldots.$ By (4.21) we get $(4.22)$ for $n=1.$ Assume that $(4.22)$ holds for some $n=n_*$, that is ,$$\|u_{n_*,\lambda}(t)\|_{L^q(w)}\leq 2C_{**}\delta,\quad 0<t\leq 2,\eqno(4.23)$$
for any $q\in[r,\infty]$. Then, by (4.23) we get $$\|u_{n_*,\lambda}(t)^p\|_{L^q(w)}=\|u_{n_*,\lambda}(t)\|^p_{L^{pq}(w)}\leq (2C_{**}\delta)^p,\quad 0<t\leq 2,\eqno(4.24)$$for any $q\in [r,\infty]$.
Choosing a sufficiently small $\delta$ if necessary, by Lemma 2.3, (4.20), (4.21) and $(4.24)$ we deduce
\begin{align*}
\|u_{n_*+1,\lambda}(t)\|_{L^q(w)}&\leq \|S(t)u_{0,\lambda}w\|_{L^q(w)}+\int^t_0\|S(t-s)u_{n_*,\lambda}(s)^pw\|_{L^q(w)}ds\\\tag{4.25}
&\leq \|S(t)u_{0,\lambda}w\|_{L^q(w)}+C_1\int^t_0\|u_{n_*,\lambda}(s)^p\|_{L^q(w)}ds\\
&\leq C_{**}\delta+C_2\delta^p\leq 2C_{**}\delta,
\end{align*}
for all $0<t\leq 2$ and for any $q\in [r,\infty]$, where $C_1$ and $C_2$ are constants independent of $n_*$ and $\delta.$ Hence we have $(4.22)$ for $n=n_*+1$, and $(4.22)$ holds for all $n=1,2,\ldots$.

Let $C_*$ be a constant to be chosen later such that $C_*\geq 2C_{**}$. By induction we prove that $$\|u_{n,\lambda}(t)\|_{L^q(w)}\leq C_*\delta t^{-\frac{n+\alpha}{2}\left(\frac{1}{r}-\frac{1}{q}\right)},\quad t>1/2,\eqno(4.26)$$
for any $q\in [r,\infty]$ and $n=1,2,\ldots$. By (4.21) we get $(4.26)$ for $n=1$. Assume that $(4.26)$ holds for some $n=n_*$. Then, similarly to (4.25), since $r_*=\frac{n+\alpha}{2}(p-1)>r,$ choosing a sufficiently small  $\delta$ if necessary, by Lemma 2.3, (4.20) and $(4.22)$ we deduce
\begin{align*}
\|u_{n_*+1,\lambda}(t)\|_{L^q(w)}&\leq C_3(t-1/2)^{-\frac{n+\alpha}{2}\left(\frac{1}{r}-\frac{1}{q}\right)}\|u_{n_*+1,\lambda}\left(1/2\right)\|_{L^r(w)}\\
&\quad +C_3\int^{\frac{t}{2}}_{\frac{1}{2}}(t-s)^{-\frac{n+\alpha}{2}\left(\frac{1}{r}-\frac{1}{q}\right)}\|u_{n_*,\lambda}(s)^p\|_{L^r(w)}ds\\
&\quad +C_3\int^t_{\frac{t}{2}}\|u_{n_*,\lambda}(s)^p\|_{L^q(w)}ds\\
&\leq C_4C_{**}\delta t^{-\frac{n+\alpha}{2}\left(\frac{1}{r}-\frac{1}{q}\right)}+C_4(C_*\delta)^p t^{-\frac{n+\alpha}{2}\left(\frac{1}{r}-\frac{1}{q}\right)}\int^\frac{t}{2}_{\frac{1}{2}}s^{-\frac{r_*}{r}}ds\\
&\quad +C_4(C_*\delta)^p\int^t_{\frac{t}{2}}s^{-\frac{n+\alpha}{2}\left(\frac{p}{r}-\frac{1}{q}\right)}ds\\
&\leq C_5C_{**}\delta t^{-\frac{n+\alpha}{2}\left(\frac{1}{r}-\frac{1}{q}\right)}+C_5(C_*\delta)^pt^{-\frac{n+\alpha}{2}\left(\frac{1}{r}-\frac{1}{q}\right)}
\end{align*}
for all $t>1,$ where $C_3$, $C_4$ and $C_5$ are constants independent of $n_*$ and $\delta$. Let $C_*\geq 2C_5C_{**}$.
Then, choosing a sufficiently small $\delta$ if necessary, we get $$\|u_{n_*+1,\lambda}(t)\|_{L^q(w)}\leq C_*\delta t^{-\frac{n+\alpha}{2}\left(\frac{1}{r}-\frac{1}{q}\right)},\quad t>1.$$
This together with (4.23) implies (4.26) with $n=n_*+1$. Hence (4.26) holds for all $n=1,2,\ldots$.

It follows from (4.22) and $(4.26)$ that there exists a constant $C$ such that $$\|u_{n,\lambda}(t)\|_{L^q(w)}\leq C\delta(1+t)^{-\frac{n+\alpha}{2}\left(\frac{1}{r}-\frac{1}{q}\right)},\quad t>0,$$
for any $q\in [r,\infty]$ and $n=1,2,\ldots$. This implies that $$\|u_n(t)\|_{L^q(w)}\leq C(1+t)^{-\frac{n+\alpha}{2}\left(\frac{1}{r}-\frac{1}{q}\right)},\quad t>0,$$
for any $q\in [r,\infty]$ and $n=1,2,\ldots$. By the same argument as in the proof of the assertion $\mathrm{(i)}$ of Theorem 1.2, we see that there exists a solution $u$ of (1.3) satisfying (1.16). Therefore the assertion $\mathrm{(ii)}$ of Theorem 1.2 follows, and the proof of Theorem 1.2 is complete.\qed
\\
\\
{\bf{Proof of Corollary 1.1}} Since $r_*=\frac{n+\alpha}{2}(p-1),$ by $(1.8)$ and (1.17) we can find a constant $C_1$ independent of $\delta$ such that $$\|u_0\|_{L^{r_*,\infty}(w)}\leq C_1\delta.$$
Hence, by the assertion $\mathrm{(i)}$ of Theorem 1.2 we have that, if $\delta$ is sufficiently small, then a global-in-time solution of (1.3) exists and it satisfies (1.14). Therefore the proof of Corollary 1.1 is complete. \qed

LMAM, School of Mathematical Sciences, Peking University, Beijing, 100871, P. R. China

Xi Hu,\quad
E-mail address: huximath1994@163.com

Lin Tang,\quad
E-mail address: tanglin@math.pku.edu.cn
\end{document}